\documentclass[11pt]{amsart}
\usepackage{lipsum}
\makeatletter
\g@addto@macro{\endabstract}{\@setabstract}
\newcommand{\authorfootnotes}{\renewcommand\thefootnote{\@fnsymbol\c@footnote}}%
\makeatother

\usepackage[T1]{fontenc}
\usepackage{tipa}
\usepackage{enumitem}
\usepackage[dvipsnames]{xcolor}
\usepackage{setspace}
\usepackage[english]{babel}
\usepackage{amsmath,amsfonts,amssymb,amsthm,epsfig,epstopdf,url,array}
\usepackage[margin=1.3in]{geometry}
\usepackage{tikz-cd}
\usepackage{calc}
\usepackage{color,hyperref}
\usepackage[noabbrev]{cleveref}
\hypersetup{colorlinks}
\hypersetup{colorlinks={true},linkcolor={red},citecolor=green}
\usepackage{graphicx}
\usepackage{mathrsfs}
\graphicspath{{images/}}
\theoremstyle{plain}
\newtheorem{theorem}{Theorem}

\newtheorem{thm}{Theorem}[subsection]
\newtheorem{lem}[thm]{Lemma}
\newtheorem{prop}[thm]{Proposition}

\theoremstyle{plain}
\newtheorem{defn}[thm]{Definition}

\theoremstyle{remark}
\newtheorem{rem}[thm]{Remark}

\newcommand{\mycomment}[1]{}
\usepackage{tikz}
\usepackage{mathtools}

\makeatletter

\renewcommand{\tocsection}[3]{%
  \indentlabel{\@ifnotempty{#2}{\bfseries\ignorespaces#1 #2\quad}}\bfseries#3}
\renewcommand{\tocsubsection}[3]{%
  \indentlabel{\@ifnotempty{#2}{\ignorespaces#1 #2\quad}}#3}

\renewcommand{\tocsubsubsection}[3]{%
  \indentlabel{\@ifnotempty{#2}{\ignorespaces#1 #2\quad}}#3}

\newcommand\@dotsep{4.5}
\def\@tocline#1#2#3#4#5#6#7{\relax
  \ifnum #1>\c@tocdepth 
  \else
    \par \addpenalty\@secpenalty\addvspace{#2}%
    \begingroup \hyphenpenalty\@M
    \@ifempty{#4}{%
      \@tempdima\csname r@tocindent\number#1\endcsname\relax
    }{%
      \@tempdima#4\relax
    }%
    \parindent\z@ \leftskip#3\relax \advance\leftskip\@tempdima\relax
    \rightskip\@pnumwidth plus1em \parfillskip-\@pnumwidth
    #5\leavevmode\hskip-\@tempdima{#6}\nobreak
    \leaders\hbox{$\m@th\mkern \@dotsep mu\hbox{.}\mkern \@dotsep mu$}\hfill
    \nobreak
    \hbox to\@pnumwidth{\@tocpagenum{\ifnum#1=1\bfseries\fi#7}}\par
    \nobreak
    \endgroup
  \fi}
\AtBeginDocument{%
\expandafter\renewcommand\csname r@tocindent0\endcsname{0pt}
}
\def\l@subsection{\@tocline{2}{0pt}{2.5pc}{5pc}{}}
\makeatother

\usepackage{xcolor}
\usepackage[new]{old-arrows}
\usepackage{bm} 
\usepackage{multicol}
\usetikzlibrary{decorations.pathmorphing}
\DeclareMathOperator{\GL}{GL}

\usepackage{mathtools}

\def\lf{\left\lfloor}   
\def\rf{\right\rfloor}

\newcommand\blfootnote[1]{%
  \begingroup
  \renewcommand\thefootnote{}\footnote{#1}%
  \addtocounter{footnote}{-1}%
  \endgroup
}

\date{} 

\begin{document}
\hypersetup{citecolor=blue}
\hypersetup{linkcolor=red}

\begin{center}
  \LARGE 
Local newforms, Pa\v{s}k\={u}nas-Stevens Whittaker functions, \\ and explicit test vectors for $\GL(n)\times \GL(n)$
  \normalsize
  \\
  \bigskip
  Alexandros Groutides \par 

  Instituto Matemáticas de la Universidad de Sevilla\par
\end{center}

\begin{abstract}
   Let $F/\mathbf{Q}_p$ be a finite extension. Recently, Girsch-Kurinczuk have obtained new formulas for local newforms in depth-zero and minimax integral-depth cuspidal representations of $\GL_n(F)$. One of their main results is an averaging formula which provides an integral representation of the Whittaker newform in terms of the Pa\v{s}k\={u}nas-Stevens Whittaker function. In the present paper, we first prove a reverse averaging formula by establishing an integral representation of the Pa\v{s}k\={u}nas-Stevens Whittaker function in terms of the Whittaker newform. Together, these formulas give an explicit two-way passage between these two distinguished vectors. 
   
   We then use our averaging formula and earlier work of Kurinczuk-Matringe to give new explicit test vectors in terms of Whittaker newforms for cuspidal $\GL(n)\times\GL(n)$ pairs. More precisely, we show that if $\pi_1,\pi_2$ are depth-zero or minimax integral-depth cuspidal representations of $\GL_n(F)$, with $L(\pi_1,\pi_2,s)\neq 1$, then
$$L(\pi_1,\pi_2,s)=Z(W_{\pi_1}^\mathrm{new},\widetilde{W}_{\pi_2}^{\mathrm{new}},\mathbf{1}_{\mathcal{O}_{F}^n\Sigma},s),$$
  where $Z$ is the $\GL(n)\times \GL(n)$ local zeta integral, $W_{\pi_i}^\mathrm{new}$ is the Whittaker newform, $\widetilde{W}_{\pi_2}^\mathrm{new}$ is an explicit average of the Whittaker newform, and $\Sigma$ is a simple diagonal matrix depending on the conductor of $\pi_i.$
\end{abstract}
\blfootnote{Work on this paper began while the author was funded by ERC101001051: Shimura varieties and the
Birch–Swinnerton-Dyer conjecture. 

This paper was completed while the author was funded by ATR2024-154613: La correspondencia de Langlands local y el dual unitario de los grupos $p$-ádicos clásicos.}
\vspace{-2em}
\section{Introduction}
\subsection{Background} Classical newform theory was first developed in the work of Atkin-Lehner \cite{AtkinLehner}. Casselman \cite{casselman1973some} later gave its local representation-theoretic formulation for $\GL_2(F)$, and Jacquet-Piatetski-Shapiro-Shalika \cite{Jacquet1981}, Jacquet \cite{jacquet2009archimedean}, and Matringe \cite{Matringeessential} extended it to all irreducible generic complex representations of $\GL_n(F)$, where $F$ is a finite extension of $\mathbf{Q}_p$. One of the main results of this theory is the existence and uniqueness of certain distinguished vectors called \textit{local newforms}. More precisely, if $K_n(m)$ denotes the open compact subgroup of $\GL_n(\mathcal{O}_F)$ whose bottom row is congruent to $(0,\dots,0,1)$ modulo $\varpi_F^m$, for $m\in\mathbf{Z}_{\geq 0}$, and if $\pi$ is an irreducible generic complex representation of $\GL_n(F)$, then there exists a unique integer $c(\pi)\in\mathbf{Z}_{\geq 0}$ called the conductor of $\pi$, such that
\begin{align}
    \pi^{K_n(m)} \simeq
    \begin{dcases*}
        \mathbf{C}, & if $m=c(\pi)$\\
        0,          & if $m<c(\pi)$.
    \end{dcases*}
\end{align}
The \textit{newform} of $\pi$ is the unique, up to scalars, non-zero vector in the one-dimensional space $\pi^{K_n(c(\pi))}$. If $\pi$ is unramified, then $c(\pi)=0$, and the newform is nothing more than the spherical vector. 

These distinguished elements and their Whittaker realizations often turn out to be \textit{test vectors}. In particular, by a result of Kim \cite{kim2010test}, if $\pi_1,\pi_2$ are generic representations of $\GL_n(F)$ with at least one of them unramified, then the pair of \textit{Whittaker newforms}, together with an appropriate Schwartz function, represents the $\GL(n)\times\GL(n)$ Rankin-Selberg $L$-factor $L(\pi_1,\pi_2,s)$ under the $\GL(n)\times\GL(n)$ local zeta integral defined by Jacquet-Piatetski-Shapiro-Shalika \cite{jacquet1983rankin}. The situation is considerably subtler when both representations are ramified. In this situation the Whittaker newforms are not, in general, known to be test vectors for the $\GL(n)\times\GL(n)$ zeta integral.

On the other hand, if one considers cuspidal representations of $\GL_n(F)$, then another distinguished vector arises, this time from the theory of types. This theory, developed by Bushnell-Kutzko \cite{Bushnel}, realizes cuspidal representations of $\GL_n(F)$ by compact induction from maximal simple types. Using this realization, Pa\v{s}k\={u}nas-Stevens \cite{PaskunasStevens2008} construct an explicit Whittaker function with small support, which is the unique, up to scalars, non-zero element of a certain isotypic component of the underlying cuspidal representation. We refer to this distinguished element as the \textit{Pa\v{s}k\={u}nas-Stevens Whittaker function}. These special Whittaker functions also turn out to be test vectors. In particular, by a result of Kurinczuk-Matringe \cite{TestvectorsRobNadir}, if $\pi_1,\pi_2$ are cuspidal representations of $\GL_n(F)$ with $L(\pi_1,\pi_2,s)$ not identically equal to $1$ - in other words $\pi_2$ is an unramified twist of $\pi_1^\vee$ - then the pair of Pa\v{s}k\={u}nas-Stevens Whittaker functions, together with an appropriate Schwartz function, represents the Rankin-Selberg $L$-factor $L(\pi_1,\pi_2,s)$. Hence, these Whittaker functions resolve the explicit test vector problem for such cuspidal pairs. However, they do not give information on whether one can obtain explicit test vectors in terms of Whittaker newforms.

Recent work of Girsch-Kurinczuk \cite{girsch2025newformscuspidalrepresentations} begins to explore the connection between these two special vectors arising in cuspidal representations. For depth-zero and minimax integral-depth cuspidal representations of $\GL_n(F)$, one of the main results of Girsch-Kurinczuk \cite{girsch2025newformscuspidalrepresentations} is an averaging formula which provides an integral representation of the Whittaker newform in terms of the Pa\v{s}k\={u}nas-Stevens Whittaker function (see \Cref{thm new in terms of PS} and \Cref{thm: new interms of PS minimax}).
\subsection{Main results}
The first purpose of the present paper is to prove a reverse averaging formula to the one in \cite{girsch2025newformscuspidalrepresentations} and thus provide an explicit two-way passage between the type-theoretic Pa\v{s}k\={u}nas-Stevens Whittaker function and the globally motivated Whittaker newform.

\subsubsection{Averaging formulas}\label{sec: averaging formulas}
Let $\psi_0$ be a non-trivial additive character of $F$ of conductor $\mathcal{O}_F$, and let $\psi_1$ be the associated conductor-$\varpi_F\mathcal{O}_F$ additive character given by $x\mapsto\psi_0(\varpi_F^{-1}x)$. Let $\Sigma_n:=\mathrm{diag}(\varpi_F^{n-1},\varpi_F^{n-2},\dots,\varpi_F,1)\in\GL_n(F)$, let $K_n^1$ be the first principal congruence subgroup of $\GL_n(\mathcal{O}_F)$ given by the kernel of reduction modulo $\varpi_F$, and let $U_n$ denote the standard upper unipotent radical of $\GL_n$. Then, the character $\psi_1$ extends to a well-defined character of the group $U_n(\mathcal{O}_F)K_n^1$. Our depth-zero averaging formula is as follows.
\begin{theorem}\emph{(Depth zero, \Cref{thm PS in terms of new})}
     Let $\pi$ be a depth-zero irreducible cuspidal representation of $\GL_n(F)$. Then, the Pa\v{s}k\={u}nas-Stevens Whittaker function $W_{\pi,\psi_1}^\mathrm{PS}\in \mathcal{W}(\pi,\psi_1)$ admits the following integral representation in terms of the Whittaker newform $W_{\pi,\psi_0}^\mathrm{new}\in\mathcal{W}(\pi,\psi_0)$:
    $$W_{\pi,\psi_1}^\mathrm{PS}(g)=\int_{U_n(\mathcal{O}_F)K_n^1} W_{\pi,\psi_0}^\mathrm{new}(\Sigma_n^{-1}gx \Sigma_n) \psi_1(x^{-1})\ dx,\ \mathrm{for\ all}\ g\in \GL_n(F),$$
    for an appropriately normalized Haar measure $dx$ on $U_n(\mathcal{O}_F)K_n^1.$
\end{theorem}
To state our positive-depth averaging formula, we recall that every minimax integral-depth cuspidal representation of $\GL_n(F)$ is constructed from a simple stratum $[\Lambda,m,0,\beta]$ and has a Bushnell-Kutzko model as a compactly induced representation from a compact-mod-center subgroup $\mathbf{J}:=E^\times U^{\lf (m+1)/2 \rf}(\Lambda)$, where $E=F[\beta]$ is a degree-$n$ unramified field extension and $U^{\lf (m+1)/2 \rf}(\Lambda)=1+M_{n,n}(\varpi_F^{\lf (m+1)/2 \rf}\mathcal{O}_F)$ with respect to a certain basis $\mathcal{B}'$ of $E/F$. For more details on these, we refer the reader to \Cref{sec: hered. order}, \Cref{sec: setup}, and \Cref{sec: simple chars}. Let $\mathbf{t}_m:=\mathrm{diag}(\varpi_F^{(n-1)m},\varpi_F^{(n-2)m},\dots, \varpi_F^m,1)\in\GL_n(F)$, $\Sigma_{m,n}:=\mathbf{t}_m\Sigma_n\in \GL_n(F)$, and $\Psi_m:=\psi_1^{\mathbf{t}_m^{-1}}$ regarded as a character of $U_n(F)$ as in \Cref{sec: add. chars}. Our positive-depth averaging formula is as follows.
\begin{theorem}\emph{(Positive depth, \Cref{thm: PS in terms of new minimax})}
     Let $\Lambda,m,\beta,\mathcal{B}'$ be as in \emph{\Cref{sec: setup}}, and the character $\theta_{\Psi_m}$ be as in \emph{\Cref{sec: simple chars}}. Let $\pi=\mathrm{c}\text{-}\mathrm{Ind}_{\mathbf{J}}^{\GL_n(F)}(\lambda)$ be an irreducible cuspidal minimax representation of $\GL_n(F)$ with integral depth, constructed as in \emph{\Cref{sec: simple chars}}. The Pa\v{s}k\={u}nas-Stevens Whittaker function $W_{\pi,\Psi_m}^\mathrm{PS}\in \mathcal{W}(\pi,\Psi_m)$ admits the following integral representation in terms of the Whittaker newform $W_{\pi,\psi_0}^\mathrm{new}\in \mathcal{W}(\pi,\psi_0)$:
    $$W_{\pi,\Psi_m}^\mathrm{PS}(g)=\int_{(U_n(F)\cap J^1)H^1}W_{\pi,\psi_0}^\mathrm{new}(\Sigma_{m,n}^{-1} gx \Sigma_{m,n})\theta_{\Psi_m}(x^{-1})\ dx,\ \mathrm{for\ all}\ g\in \GL_n(F),$$
    for an appropriately normalized Haar measure $dx$ on $(U_n(F)\cap J^1)H^1$.
\end{theorem}
\subsubsection{Explicit test vectors}
Now, consider a pair $\pi_1,\pi_2$ of cuspidal representations of $\GL_n(F)$. The $\GL(n)\times\GL(n)$ local zeta integral of Jacquet-Piatetski-Shapiro-Shalika \cite{jacquet1983rankin} is given by 
\begin{align}\label{eq: zeta integral intro}
Z(W_1,W_2,\Phi,s):=\int_{U_n(F)\backslash\GL_n(F)} W_1(g)W_2(g)\Phi((0,\dots,0,1)g)|\det(g)|^s\ dg
\end{align}
for $W_1\in\mathcal{W}(\pi_1,\psi_0)$, $W_2\in\mathcal{W}(\pi_2,\psi_0^{-1})$, $\Phi\in\mathcal{S}(F^n)$.

The second purpose of this paper is to give explicit test vectors in terms of Whittaker newforms, which, under \eqref{eq: zeta integral intro}, represent the $L$-factor $L(\pi_1,\pi_2,s)$ whenever it is not identically equal to $1$ - in other words, whenever $\pi_2$ is an unramified twist of $\pi_1^\vee$. The main ingredients for this are our averaging formulas from \Cref{sec: averaging formulas} and the main result of Kurinczuk-Matringe \cite{TestvectorsRobNadir}.

\begin{theorem}\emph{(Depth zero, \Cref{thm test vector depth zero})}
 Let $\pi$ be a depth-zero irreducible cuspidal representation of $\GL_n(F)$, and $\chi$ an unramified character of $F^\times$. Write $\chi\pi$ for the twist $\chi(\det(-))\otimes\pi$. Consider the following input data:
    \begin{itemize}
        \item The Whittaker newform $W_{\chi \pi,\psi_0}^\mathrm{new}\in \mathcal{W}(\chi\pi,\psi_0)$,
        \item The average $\widetilde{W}_{\pi^\vee,\psi_0^{-1}}^\mathrm{new}\in \mathcal{W}(\pi^\vee,\psi_0^{-1})$ of the dual Whittaker newform, 
$$\widetilde{W}_{\pi^\vee,\psi_0^{-1}}^\mathrm{new}:=\int_{U_n(\mathcal{O}_F)K_n^1} W_{\pi^\vee,\psi_0^{-1}}^\mathrm{new}( (-)\Sigma_n^{-1}y\Sigma_n) \psi_1(y)\ dy,$$
        \item The Schwartz function $\Phi_n:=\mathbf{1}_{\mathcal{O}_F^n\cdot \Sigma_n}\in \mathcal{S}(F^n)$.
    \end{itemize}
    Then, up to appropriate normalization of the quotient measure $dg$ appearing in \eqref{eq: zeta integral intro}, we have 
    $$L(\chi\pi,\pi^\vee,s)=Z(W_{\chi \pi,\psi_0}^\mathrm{new},\widetilde{W}_{\pi^\vee,\psi_0^{-1}}^\mathrm{new}, \Phi_n,s). $$
\end{theorem}
\begin{theorem}\emph{(Positive depth, \Cref{thm: test vector pos. depth})}
    Let $\Lambda,m,\beta,\mathcal{B}'$ be as in \emph{\Cref{sec: setup}}, and the character $\theta_{\Psi_m}$ be as in \emph{\Cref{sec: simple chars}}. Let $\pi=\mathrm{c}\text{-}\mathrm{Ind}_{\mathbf{J}}^{\GL_n(F)}(\lambda)$ be an irreducible cuspidal minimax representation of $\GL_n(F)$ with integral depth, constructed as in \emph{\Cref{sec: simple chars}}. Let $\chi$ be an unramified character of $F^\times$ and consider the following input data:
 \begin{itemize}
        \item The Whittaker newform $W_{\chi \pi,\psi_0}^\mathrm{new}\in \mathcal{W}(\chi\pi,\psi_0)$,
        \item The average $\widetilde{W}_{\pi^\vee,\psi_0^{-1}}^\mathrm{new}\in \mathcal{W}(\pi^\vee,\psi_0^{-1})$ of the dual Whittaker newform, 
$$\widetilde{W}_{\pi^\vee,\psi_0^{-1}}^\mathrm{new}:=\int_{(U_n(F)\cap J^1)H^1} W_{\pi^\vee,\psi_0^{-1}}^\mathrm{new}( (-)\Sigma_{m,n}^{-1}y\Sigma_{m,n}) \theta_{\Psi_m}(y)\ dy,$$
        \item The Schwartz function $\Phi_{m,n}:=\mathbf{1}_{\mathcal{O}_F^n\cdot \Sigma_{m,n}}\in \mathcal{S}(F^n)$.
    \end{itemize}
    Then, up to appropriate normalization of the quotient measure $dg$ appearing in \eqref{eq: zeta integral intro}, we have 
    $$L(\chi\pi,\pi^\vee,s)=Z(W_{\chi \pi,\psi_0}^\mathrm{new},\widetilde{W}_{\pi^\vee,\psi_0^{-1}}^\mathrm{new}, \Phi_{m,n},s).$$
\end{theorem}
\subsection{Acknowledgments}
I would like to thank Robert Kurinczuk for various helpful discussions and for encouraging me to pursue this project. I would also like to thank Alberto Minguez for useful comments regarding this paper.
\subsection{Declaration on the use of AI}
During the preparation of this work, the author used AI solely to check and correct spelling, formatting, and grammatical errors. The author takes full responsibility for the contents of this paper.
{
  \hypersetup{linkcolor=black}
  \setcounter{tocdepth}{1}
  \tableofcontents
}
\section{Background}
Let $F$ be a finite extension of $\mathbf{Q}_p$ with ring of integers $\mathcal{O}_F$ and a choice of uniformizer $\varpi_F$. Let $\mathfrak{p}_F:=\varpi_F\mathcal{O}_F$ be the maximal ideal of $\mathcal{O}_F$ and $k_F:=\mathcal{O}_F/\mathfrak{p}_F$ the residue field. By a representation of a locally profinite group, we mean a smooth representation on a complex vector space.
\subsection{Groups}
For an integer $n\in \mathbf{Z}_{\geq 1}$, we define the groups:
\[
\begin{aligned}
G_n   &:= \GL_n(F),
&\qquad Z_n &:= Z(G_n)\simeq F^\times, \\[6pt]
K_n &:= \GL_n(\mathcal{O}_F),
&\qquad K_n^1 &:= \ker\bigl(
  K_n \overset{\mathrm{mod}\ \varpi_F}{\longrightarrow}
  \GL_n(k_F)\bigr).
\end{aligned}
\]
We write $B_n$ for the upper triangular Borel subgroup of $\GL_n$ and $U_n$ for its unipotent radical. We denote by $B_n^\mathrm{op}$ the oppositive Borel subgroup of lower triangular matrices. For $n\in \mathbf{Z}_{\geq 2}$, whenever we regard $G_{n-1}$ as a subgroup of $G_n$ we do so via the block diagonal embedding $G_{n-1}\hookrightarrow G_n,\ g\mapsto \mathrm{diag}(g,1)$. For an additive group $\Gamma$ and integers $a,b\in \mathbf{Z}_{\geq 1}$, we write $M_{a,b}(\Gamma)$ for the additive group of $a\times b$ matrices with coefficients in $\Gamma$.
\subsection{Conjugation matrices}\label{sec: conj matrices} For integers $n\geq 2,m\geq 1$, we define the diagonal matrices 
\begin{align*}\Sigma_n&:=\mathrm{diag}(\varpi_F^{n-1},\varpi_F^{n-2},\dots, \varpi_F, 1)\in G_n,\\
\mathbf{t}_{m}&:=\mathrm{diag}(\varpi_F^{(n-1)m},\varpi_F^{(n-2)m},\dots, \varpi_F^m,1)\in G_n,\\
\Sigma_{m,n}&:=\mathbf{t}_{m}\Sigma_n\in G_n.
\end{align*}
\subsection{Additive characters}\label{sec: add. chars}
Let $\psi_0:F\rightarrow \mathbf{C}^\times$ be a non-trivial additive character of conductor $\mathcal{O}_F$, i.e. $\psi_0$ is trivial on $\mathcal{O}_F$ but non-trivial on $\mathfrak{p}_F^{-1}$. Let $\psi_1:F\rightarrow \mathbf{C}^\times$ be the associated non-trivial additive character given by the assignment $x\mapsto \psi_0(\varpi_F^{-1} x)$. Then, $\psi_1$ has conductor $\mathfrak{p}_F$, i.e. it is trivial on $\mathfrak{p}_F$ but non-trivial on $\mathcal{O}_F$. We extend the non-trivial characters $\psi_0$ and $\psi_1$ to non-degenerate characters of $U_n(F)$ by
\begin{align*}
    \psi_0&:U_n(F)\rightarrow \mathbf{C}^\times,\ u\mapsto \psi_0(u_{1,2}+\dots u_{n-1,n}),\\
    \psi_1&:U_n(F)\rightarrow \mathbf{C}^\times,\ u\mapsto \psi_1(u_{1,2}+\dots u_{n-1,n}).
\end{align*}
Then, as characters of $U_n(F)$, we have 
\begin{align}\label{eq: character relations}
\psi_0=\psi_1^{\Sigma_n}:=\psi_1(\Sigma_n(-)\Sigma_n^{-1}),\ \ \ \psi_1=\psi_0^{\Sigma_n^{-1}}:=\psi_0(\Sigma_n^{-1}(-)\Sigma_n).
\end{align}
We also define the non-degenerate character 
\begin{align}
\Psi_m:=\psi_1^{\mathbf{t}_m^{-1}}:=\psi_1(\mathbf{t}_{m}^{-1}(-)\mathbf{t}_{m}): U_n(F)\rightarrow \mathbf{C}^\times.
\end{align}
The character $\Psi_m$ will not appear until \Cref{sec: pos. depth}.
\subsection{Bessel functions for $\GL_n(k_F)$}\label{sec modular Bessel functions}
The character $\psi_1$ induces a non-trivial additive character $\overline{\psi_1}: k_F\rightarrow \mathbf{C}^\times.$ As before, we extend this to a non-degenerate character of $U_n(k_F)$ by 
$$\overline{\psi_1}:U_n(k_F)\rightarrow \mathbf{C}^\times,\ \  u\mapsto \overline{\psi_1}(u_{1,2}+\dots + u_{n-1,n}).$$
Let $\tau$ be an irreducible cuspidal representation of $\GL_n(k_F)$. Such a representation is always $\overline{\psi_1}$-generic and thus admits a non-zero Whittaker model $\mathcal{W}(\tau,\overline{\psi_1})$ isomorphic to $\tau$. 
\begin{defn}\label{defn Bessel function} The Bessel function $B_{\tau,\overline{\psi_1}} \in \mathcal{W}(\tau,\overline{\psi_1})$ is the unique function in $\mathcal{W}(\tau,\overline{\psi_1})$ which satisfies:
\begin{enumerate}
    \item $B_{\tau,\overline{\psi_1}}(1)=1$,
    \item $B_{\tau,\overline{\psi_1}}(u_1gu_2)=\overline{\psi_1}(u_1)B_{\tau,\overline{\psi_1}}(g) \overline{\psi_1}(u_2)$ for all $u_i\in U_n(k_F)$ and $g\in \GL_n(k_F).$
\end{enumerate}
\end{defn}

\subsection{Newform theory}
Let $n\in \mathbf{Z}_{\geq 2}$ and $m\in \mathbf{Z}_{\geq 0}$. We define the following open compact subgroup of $K_n$
$$K_n(m):=\left\{\left[\begin{matrix}
    A & B \\
    C & D
\end{matrix}\right]\in K_n\ |\ C\in M_{1,(n-1)}(\mathfrak{p}_F^m),\ D\in 1+ \mathfrak{p}^m\right\}.$$
Thus, $K_n(0)=K_n$.

\begin{defn}
    Let $\pi$ be an irreducible generic representation of $G_n$.
    \begin{enumerate}
        \item If there exists an integer $m\in\mathbf{Z}_{\geq 0}$ such that $\pi^{K_n(m)}\neq 0$, we define the \textit{conductor} $c(\pi)\in \mathbf{Z}_{\geq 0}$ to be the minimal such $m$.
        \item We call a non-zero vector in $\pi^{K_n(c(\pi))}$ a \textit{newform}.
    \end{enumerate}
\end{defn}
In their seminal work, Jacquet-Shalika-Piatetski-Shapiro established uniqueness and existence of newforms, generalizing Casselman's work from $n=2$ to arbitrary $n\geq2.$ 
\begin{thm}[\cite{Jacquet1981}, \cite{correction}]\label{thm: existence uniequness of newforms over C}
    Let $\pi$ be an irreducible generic representation of $G_n$.
    \begin{enumerate}
        \item The conductor $c(\pi)$ exists.
        \item $\mathrm{dim}_\mathbf{C}\ \pi^{K_n(c(\pi))}=1$.
    \end{enumerate}
\end{thm}
Moreover, by a result of Bushnell \cite{Bushnel}, given a cuspidal representation $\pi$ of $G_n$, we have the formula 
\begin{align}\label{eq: Bushnel conductor formula}c(\pi)=n(1+d(\pi)),\end{align}
where $d(\pi)$ denotes the \textit{depth} of $\pi$. In particular, if $d(\pi)=0$ then $c(\pi)=n.$ 

Since $\pi$ is assumed generic, it has a $\psi_0$-Whittaker model $\mathcal{W}(\pi,\psi_0)\simeq\pi$. By the above theorem, the space $\mathcal{W}(\pi,\psi_0)^{K_n(c(\pi))}\simeq \mathbf{C}$. We write $W_{\pi,\psi_0}^\mathrm{new}$ for the unique, up to scalars, non-zero element of $\mathcal{W}(\pi,\psi_0)^{K_n(c(\pi))}$ and we refer to it as the \textit{Whittaker newform} of the representation $\pi.$ It is completely determined in the Kirillov model of $\pi$ by the Matringe-Miyauchi formula \cite{Matringeessential}, \cite{Miauchi}, which in particular implies that $W_{\pi,\psi_0}^\mathrm{new}(1)\neq 0.$
\section{Integral representations in depth zero}\label{sec: depth zero}
\subsection{Bessel vectors for $\GL_n(F)$ in depth zero}\label{sec Bessel vector} Let $\pi$ be a depth-zero irreducible cuspidal representation of $G_n$ with central character $\omega_\pi$. We choose a Moy-Prasad-Morris-Vign\'eras model \cite[III, $3.3$]{vigneras1996representations} 
$$\pi\simeq \mathrm{c}\text{-}\mathrm{Ind}_{Z_n K_n}^{G_n}\widetilde{\tau}$$
where $\widetilde{\tau}:=\omega_\pi\tau$ with $\tau\simeq \mathcal{W}(\tau,\overline{\psi_1})$ an irreducible cuspidal representation of $\GL_n(k_F)$ inflated to a representation of $K_n$. 
\begin{defn}
    The Bessel vector $\mathcal{B}_{\widetilde{\tau},\psi_1}\in \pi\simeq\mathrm{c}\text{-}\mathrm{Ind}_{Z_n K_n}^{G_n}\widetilde{\tau}$ is defined by $\mathrm{supp}(\mathcal{B}_{\widetilde{\tau},\psi_1})\subseteq Z_nK_n$ and 
    $$\mathcal{B}_{\widetilde{\tau},\psi_1}(zk)=\omega_\pi(z)k\cdot B_{\tau,\overline{\psi_1}},\ \mathrm{for\ all}\ z\in Z_n, k\in K_n,$$
    where $B_{\tau,\overline{\psi_1}}\in \mathcal{W}(\tau,\overline{\psi_1})$ is the Bessel function of \emph{\Cref{defn Bessel function}}.
\end{defn}
Since the additive character $\psi_1$ has conductor $\mathfrak{p}_F$, and since $U_n(\mathcal{O}_F)$ normalizes $K_n^1$, we get a well-defined character 
\begin{align}\label{eq: char Psi_1}\psi_1:U_n(\mathcal{O}_F)K_n^1\rightarrow \mathbf{C}^\times,\ \ uk\mapsto \psi_1(u),\ \ \mathrm{for\ all}\ u\in U_n(\mathcal{O}_F), k\in K_n^1.\end{align}
\begin{prop}[\cite{PaskunasStevens2008} Corollary $4.5$]\label{prop Bessel vector} The Bessel vector $\mathcal{B}_{\widetilde{\tau},\psi_1}$ is the unique, up to scalars, vector in $\pi\simeq\mathrm{c}\text{-}\mathrm{Ind}_{Z_n K_n}^{G_n}\widetilde{\tau}$ which satisfies
$$x\cdot \mathcal{B}_{\widetilde{\tau},\psi_1}=\psi_1(x) \mathcal{B}_{\widetilde{\tau},\psi_1},\ \mathrm{for\ all}\ x\in U_n(\mathcal{O}_F)K_n^1.$$
\end{prop}
Write $\mathcal{B}_{\widetilde{\tau},\psi_1}^\vee$ for the Bessel vector in $\pi^\vee\simeq \mathrm{c}\text{-}\mathrm{Ind}_{Z_n K_n}^{G_n}(\widetilde{\tau}^\vee)$.
\subsection{Pa\v{s}k\={u}nas-Stevens Whittaker functions in depth zero}\label{sec: PS in depth zero}
Let $\pi$ be a depth-zero irreducible cuspidal representation of $G_n$ identified with $\mathrm{c}\text{-}\mathrm{Ind}_{Z_n K_n}^{G_n}\widetilde{\tau}$. Write $\langle\cdot,\cdot\rangle$ for the evaluation pairing on $\pi \otimes\pi^\vee.$
\begin{thm}\emph{(\cite[Theorem $5.8$]{PaskunasStevens2008})}\label{thm PS main thm}
    Let $$\mathbf{Y}_1:\pi\rightarrow\mathrm{Ind}_{U_n(F)}^{G_n}\psi_1,\ f\mapsto \left(\mathbf{Y}_1(f): g\mapsto \int_{U_n(F)}\psi_1(u)\langle\pi(u^{-1}g)f,\mathcal{B}_{\widetilde{\tau},\psi_1}^\vee\rangle\ du\right)$$
    where $du$ is normalized by $\mathrm{vol}(U_n(\mathcal{O}_F),du)=1$. Then,
    \begin{enumerate}
        \item The map $\mathbf{Y}_1$ is non-zero, $G_n$-equivariant, and its image is the $\psi_1$-Whittaker model $\mathcal{W}(\pi,\psi_1)\simeq\pi$.
        \item If we define $W_{\pi,\psi_1}^\mathrm{PS}:=\mathbf{Y}_1(\mathcal{B}_{\widetilde{\tau},\psi_1})\in \mathcal{W}(\pi,\psi_1)$, then $\mathrm{supp}(W_{\pi,\psi_1}^\mathrm{PS})\subseteq Z_nU_n(F) K_n$, and 
        $$W_{\pi,\psi_1}^\mathrm{PS}(zuk)=\psi_1(u)\cdot \mathcal{B}_{\widetilde{\tau},\psi_1}(zk)(1),\ \mathrm{for\ all}\ z\in Z_n, u\in U_n(F), k\in K_n.$$
    \end{enumerate}
\end{thm}
We refer to $W_{\pi,\psi_1}^\mathrm{PS}\in \mathcal{W}(\pi,\psi_1)$ as the \textit{Pa\v{s}k\={u}nas-Stevens Whittaker function} of the cuspidal representation $\pi$.
\begin{rem}
    The above theorem (\Cref{thm PS main thm}) extends to positive depth cuspidal representations of $G_n$ with some modifications. We refer the reader to \cite{PaskunasStevens2008} and \Cref{sec: PS in terms of new pos depth} for more details on the positive depth case. In the depth zero case described above, the Whittaker function $W_{\pi,\psi_1}^\mathrm{PS}$ actually recovers a classical construction of Gelfand.  
\end{rem}
\subsection{Whittaker newforms as averages of Pa\v{s}k\={u}nas-Stevens Whittaker functions}
In \cite{girsch2025newformscuspidalrepresentations} Girsch-Kurinczuk study the Whittaker newform associated to a depth-zero cuspidal representation of $G_n$, and they show that it can be expressed as an average of the Pa\v{s}k\={u}nas-Stevens Whittaker function introduced in the previous section. We restate their result below.
\begin{thm}\emph{(\cite[Theorem $4.1$ \& Proposition $5.5$]{girsch2025newformscuspidalrepresentations})}\label{thm new in terms of PS}
    Let $\pi$ be a depth-zero irreducible cuspidal representation of $G_n$. Then, the Whittaker newform $W_{\pi,\psi_0}^\mathrm{new}\in\mathcal{W}(\pi,\psi_0)$ admits the following integral representation in terms of the Pa\v{s}k\={u}nas-Stevens Whittaker function $W_{\pi,\psi_1}^\mathrm{PS}\in \mathcal{W}(\pi,\psi_1)$:
        $$W_{\pi,\psi_0}^\mathrm{new}(g)=\int_{K_n(n)}W_{\pi,\psi_1}^\mathrm{PS}(\Sigma_ngk \Sigma_n^{-1})\ dk,\ \mathrm{for\ all\ }g\in G_n$$
        for an appropriately normalized Haar measure $dk$ on $K_n(n).$

\end{thm}

\subsection{Pa\v{s}k\={u}nas-Stevens Whittaker functions as averages of Whittaker newforms}\label{PS interms of new}
Our goal in the rest of this section is to establish the reverse averaging formula to the one in \Cref{thm new in terms of PS}. In the theorem below, we give an integral representation of the Pa\v{s}k\={u}nas-Stevens Whittaker function of $\pi$ in terms of the Whittaker newform of $\pi$ and the character $\psi_1:U_n(\mathcal{O}_F)K_n^1\rightarrow \mathbf{C}^\times$ given in \eqref{eq: char Psi_1}. 

\begin{thm}\label{thm PS in terms of new}
    Let $\pi$ be a depth-zero irreducible cuspidal representation of $G_n$. Then, the Pa\v{s}k\={u}nas-Stevens Whittaker function $W_{\pi,\psi_1}^\mathrm{PS}\in \mathcal{W}(\pi,\psi_1)$ admits the following integral representation in terms of the Whittaker newform $W_{\pi,\psi_0}^\mathrm{new}\in\mathcal{W}(\pi,\psi_0)$:
    $$W_{\pi,\psi_1}^\mathrm{PS}(g)=\int_{U_n(\mathcal{O}_F)K_n^1} W_{\pi,\psi_0}^\mathrm{new}(\Sigma_n^{-1}gx \Sigma_n) \psi_1(x^{-1})\ dx,\ \mathrm{for\ all}\ g\in G_n,$$
    for an appropriately normalized Haar measure $dx$ on $U_n(\mathcal{O}_F)K_n^1.$
\end{thm}

The rest of this section is devoted to proving \Cref{thm PS in terms of new}. First, define the function $\mathbf{W}:G_n\rightarrow \mathbf{C}$ by 
$$\mathbf{W}(g):= \int_{U_n(\mathcal{O}_F)K_n^1} W_{\pi,\psi_0}^\mathrm{new}(\Sigma_n^{-1}gx \Sigma_n) \psi_1(x^{-1})\ dx,\ g\in G_n.$$
\begin{lem}\label{lem: W is in Whit model}
    The function $\mathbf{W}$ is an element of $\mathcal{W}(\pi,\psi_1).$
    \begin{proof}
        By the character relations of \eqref{eq: character relations}, given any $W\in \mathcal{W}(\pi,\psi_0)$, the function on $G_n$ defined as $W'(g):=W(\Sigma_n^{-1}g)$, is an element of the $\psi_1$-Whittaker space $\mathrm{Ind}_{U_n(F)}^{G_n} \psi_1.$ Since the assignment 
        $$\pi\simeq\mathcal{W}(\pi,\psi_0)\rightarrow \mathrm{Ind}_{U_n(F)}^{G_n} \psi_1,\ W\mapsto W'$$
        is $G_n$-equivariant, the function $W'$ lands inside the $\psi_1$-Whittaker model $\mathcal{W}(\pi,\psi_1)\simeq\pi.$ Applying this with $W_0:=W_{\pi,\psi_0}^\mathrm{new}$, we see that $(x\Sigma_n)\cdot W_0'\in \mathcal{W}(\pi,\psi_1)$. Since the integrand in the definition of $\mathbf{W}$ is smooth as a function of $x\in U_n(\mathcal{O}_F)K_n^1$, and since $U_n(\mathcal{O}_F)K_n^1$ is compact, the result follows from the preceding observation.
    \end{proof}
\end{lem}
\begin{lem}\label{lem: reduction to W non zero}
    If $\mathbf{W}\neq0$, then \emph{\Cref{thm PS in terms of new}} holds.
    \begin{proof}
        By construction, the Whittaker function $\mathbf{W}$ lies in the $(U_n(\mathcal{O}_F)K_n^1,\psi_1)$-isotypic subspace of $\mathcal{W}(\pi,\psi_1)$, where $\psi_1$ is regarded as a character of $U_n(\mathcal{O}_F)K_n^1$ as in \eqref{eq: char Psi_1}. But by \Cref{prop Bessel vector} and \Cref{thm PS main thm} this isotypic subspace is one-dimensional and generated by $W_{\pi,\psi_1}^\mathrm{PS}$.
    \end{proof}
\end{lem}
\subsubsection{Non-vanishing of $\mathbf{W}(1)$}
By the above lemmas, the proof of \Cref{thm PS in terms of new} can be reduced to showing that $\mathbf{W}(1)\neq 0$. We begin by unfolding the integral defining $\mathbf{W}(1)$. Since $U_n(\mathcal{O}_F)$ normalizes $K_n^1$, an application of \cite[Lemma $2.4$]{kurinczuk2017rankin} allows us to split the integral as 
$$\mathbf{W}(1)=\int_{U_n(\mathcal{O}_F)}\int_{K_n^1} W_{\pi,\psi_0}^\mathrm{new}(\Sigma_n^{-1} uk \Sigma_n)\psi_1(u^{-1})\ du\ dk$$
for appropriately normnalized Haar measures $du$ on $U_n(\mathcal{O}_F)$ and $dk$ on $K_n^1$. Then, notice that $W_{\pi,\psi_0}^\mathrm{new}(\Sigma_n^{-1} uk \Sigma_n)=\psi_0(\Sigma_n^{-1} u \Sigma_n)W_{\pi,\psi_0}^\mathrm{new}(\Sigma_n^{-1} k \Sigma_n)=\psi_1(u) W_{\pi,\psi_0}^\mathrm{new}(\Sigma_n^{-1} k \Sigma_n)$ for any $u\in U_n(\mathcal{O}_F)$ and $k\in K_n^1$. Thus, 
\begin{align}\label{eq: integral simplified}
    \mathbf{W}(1)=\mathrm{vol}(U_n(\mathcal{O}_F),du)\cdot\int_{K_n^1} W_{\pi,\psi_0}^\mathrm{new}(\Sigma_n^{-1}k \Sigma_n)\ dk.
\end{align}
So in order to show that $\mathbf{W}(1)$ is non-zero we need to show that the integral over $K_n^1$ appearing in \eqref{eq: integral simplified} is non-zero. 

For $n=2$, we have $\Sigma_n^{-1}K_n^1 \Sigma_n \subseteq K_n$ and we can directly show via matrix manipulations that $W_{\pi,\psi_0}^\mathrm{new}(\Sigma_n^{-1} k\Sigma_n)$ is constant and non-zero on the whole of $K_n^1$. This is no longer the case for $n>2$ so we need to argue differently. We will first give the direct computation in the case $n=2$, and then proceed to give a different proof in the general case.
\subsubsection{The $\GL_2$ case}
\begin{prop}\label{prop: main GL(2) depth zero}
    Let $\pi$ be a depth-zero irreducible cuspidal representation of $G_2=\GL_2(F).$ Then, 
    $$\int_{K_2^1} W_{\pi,\psi_0}^\mathrm{new}(\Sigma_2^{-1} k \Sigma_2) \ dk \neq 0.$$
    \begin{proof}
        Let $k$ be an arbitrary element of $K_2^1$ which we write as 
        $$k=\left[\begin{matrix}
            1+k_1 & k_2\\
            k_3 & 1+k_4 
        \end{matrix}\right],\ k_i\in \mathfrak{p}_F.$$
        A direct calculation shows that
        \begin{align*}
            &\Sigma_2^{-1} k\Sigma_2=\left[\begin{matrix}
                1+k_1 & \varpi_F k_2^{-1}\\
                \varpi_F k_3 & 1+k_4
            \end{matrix}\right]\\
            &=\underbrace{\left[\begin{matrix}
                1 & \varpi_F^{-1} k_2(1+k_4)^{-1} \\
                & 1
            \end{matrix}\right]}_{A}\underbrace{\left[\begin{matrix}
                1+k_4 & \\
                 & 1+k_4
            \end{matrix}\right]}_{B}\underbrace{\left[\begin{matrix}
                \det(k)(1+k_4)^{-2} & \\
                & 1
            \end{matrix}\right]}_{C}\underbrace{\left[\begin{matrix}
                1 & \\
                \varpi_Fk_3(1+k_4)^{-1}& 1
            \end{matrix}\right]}_{D}.
        \end{align*}
        Thus, $W_{\pi,\psi_0}^\mathrm{new}(\Sigma_2^{-1} k \Sigma_2)=\psi_0(A)\omega_\pi(B) W_{\pi,\psi_0}^\mathrm{new}(CD).$ Now, we notice that $\psi_0(A)=1$ since $k_2\in \mathfrak{p}_F$, and that $\omega_\pi(B)=1$ since $k_4\in\mathfrak{p}_F$ and $\pi$ has depth zero. Finally, it is clear that $C,D\in K_2(2)$ which fixes $W_{\pi,\psi_0}^\mathrm{new}$. Since $W_{\pi,\psi_0}^\mathrm{new}$ is supported at the identity by \Cref{thm new in terms of PS}$(1)$, the result follows.
    \end{proof}
\end{prop}
\begin{rem}
    Notice that the proof of \Cref{prop: main GL(2) depth zero} only uses the depth-zero assumption through the fact that $c(\omega_\pi)\leq 1$, and it does not use cuspidality at all. Thus, if $\pi$ is any irreducible generic representation of $\GL_2(F)$ with $c(\omega_\pi)\leq 1$, then the Whittaker function 
    $$W(g)=\int_{U_2(\mathcal{O}_F)K_2^1} W_{\pi,\psi_0}^\mathrm{new}(\Sigma_2^{-1}gx \Sigma_2) \psi_1(x^{-1})\ dx,\ g\in \GL_2(F)$$
    is always a non-zero element of
    $\mathcal{W}(\pi,\psi_1)^{(U_n(\mathcal{O}_F)K_2^1,\psi_1)}$. 
    
    Unlike the $n=2$ case, for general $n\geq 2$ the depth-zero and cuspidality assumptions on $\pi$ will be essential in our proof.
\end{rem}
\subsubsection{The $\GL_n$ case}
\begin{prop}\label{prop: non-vanishing for general n}
    Let $\pi$ be a depth-zero irreducible cuspidal representation of $G_n$. Then, 
    $$\int_{K_n^1}W_{\pi,\psi_0}^\mathrm{new}(\Sigma_n^{-1} k \Sigma_n)\ dk \neq 0.$$
    \begin{proof}
        Recall the intertwining map $\mathbf{Y}_1\in \mathrm{Hom}_{G_n}(\pi,\mathrm{Ind}_{U_n(F)}^{G_n}\psi_1)$ of \Cref{thm PS main thm}. By Frobenius reciprocity, we will regard $\mathbf{Y}_1$ as an element of $\mathrm{Hom}_{U_n(F)}(\pi,\psi_1).$ Then, the assignment 
        $$\mathbf{Y}_0:f\mapsto \mathbf{Y}_1(\Sigma_n \cdot f),\ f\in \pi\simeq \mathrm{c}\text{-}\mathrm{Ind}_{Z_nK_n}^{G_n}\tilde{\tau},$$
        is an element of $\mathrm{Hom}_{U_n(F)}(\pi,\psi_0).$ By uniqueness of the Whittaker model, uniqueness of the newform, and the fact that the Whittaker newform is supported at the identity, we must have that, up to normalization by a non-zero scalar, 
    $$W_{\pi_,\psi_0}^\mathrm{new}(g)=\mathbf{Y}_0(g\cdot f^\mathrm{new}), \mathrm{for\ all}\ g\in G_n,$$
    where $f^\mathrm{new}$ denotes the unique, up to scalars, newform in the model $\mathrm{c}\text{-}\mathrm{Ind}_{Z_nK_n}^{G_n}\tilde{\tau}$. Therefore, 
    \begin{align*}
        \int_{K_n^1}W_{\pi,\psi_0}^\mathrm{new}(\Sigma_n^{-1} k \Sigma_n)\ dk&=\int_{K_n^1}\mathbf{Y}_0(\Sigma_n^{-1}k \Sigma_n \cdot f^\mathrm{new})\ dk\\
        &=\int_{K_n^1} \mathbf{Y}_1(k \Sigma_n \cdot f^\mathrm{new})\ dk\\
        &=\mathbf{Y}_1(\mathbf{f})
    \end{align*}
    where
    $$\mathbf{f}:=\int_{K_n^1} k\Sigma_n\cdot f^\mathrm{new}\ dk \in \pi\simeq \mathrm{c}\text{-}\mathrm{Ind}_{Z_nK_n}^{G_n}\tilde{\tau}.$$
    By construction, the element $\mathbf{f}$ lies in the space of $K_n^1$-invariants $\pi^{K_n^1}$. We now need the following lemma regarding this space:
    \begin{lem}\label{lem K_n^1 invariants}
        With $\pi\simeq \mathrm{c}\text{-}\mathrm{Ind}_{Z_nK_n}^{G_n}\tilde{\tau}$ as above, there is a $K_n/K_n^1\simeq \GL_n(k_F)$-equivariant isomorphism of representations
        $$\iota:\pi^{K_n^1}\rightarrow \tau,\ f\mapsto f(1)$$
        with inverse given by
        $$\iota^{-1}:\tau{\rightarrow}\pi^{K_n^1},\ x\mapsto \left(f_{[x]}:g\mapsto \begin{dcases}
            \tilde{\tau}(g)x,\ &\mathrm{if}\ g\in Z_nK_n\\
            0,\ &\mathrm{if}\ g\not\in Z_nK_n
        \end{dcases}\right).$$
        \begin{proof}
            The only thing that is not immediately clear is that the support of an element in $\pi^{K_n^1}$ is contained in $Z_nK_n$. This follows by Mackey theory and cuspidality of $\tau.$ See for example the proof of \cite[Theorem $6.2$]{PrasadRaghuram2008}.
        \end{proof}
    \end{lem}
   Using the fact that $\tilde{\tau}(k)$ acts trivially for $k\in K_n^1$, we see that, up to a non-zero scalar, $\mathbf{f}(1)=f^\mathrm{new}(\Sigma_n)=f_{\left[f^\mathrm{new}(\Sigma_n)\right]}(1)$, where now $f_{\left[f^\mathrm{new}(\Sigma_n)\right]}\in \pi^{K_n^1}$ and thus by \Cref{lem K_n^1 invariants} we see that $\mathbf{f}=f_{\left[f^\mathrm{new}(\Sigma_n)\right]}$ as elements of $\pi$. Thus we are reduced to evaluating the quantity $$\mathbf{Y}_1(f_{\left[f^\mathrm{new}(\Sigma_n)\right]}).$$
   To do this, we will use the following explicit expression for $f^\mathrm{new}(\Sigma_n)$ obtained in \cite{girsch2025newformscuspidalrepresentations}:
   \begin{thm}\emph{(\cite[Theorem $4.1(3)$]{girsch2025newformscuspidalrepresentations})}\label{thm: f^new interms of Bessel}
       With $\pi\simeq \mathrm{c}\text{-}\mathrm{Ind}_{Z_nK_n}^{G_n}\tilde{\tau}$ as above, 
       $$f^\mathrm{new}(\Sigma_n)=\sum_{\overline{b}\in B_{n-1}^\mathrm{op}(k_F)}\overline{b}\cdot B_{\tau,\overline{\psi_1}}$$
       where $B_{\tau,\overline{\psi_1}}$ is the Bessel function from \emph{\Cref{sec modular Bessel functions}}.
   \end{thm}
   Combining \Cref{thm: f^new interms of Bessel} with the fact that $f_{[B_{\tau,\overline{\psi_1}}]}$ is nothing more than the Bessel vector $\mathcal{B}_{\tilde{\tau},\psi_1}$, we have that up to a non-zero scalar,
   \begin{align*}
       \mathbf{Y}_1(f_{\left[f^\mathrm{new}(\Sigma_n)\right]})&=\sum_{\overline{b}\in B_{n-1}^\mathrm{op}(k_F)}\mathbf{Y}_1(b\cdot \mathcal{B}_{\tilde{\tau},\psi_1})\\
       &=\sum_{\overline{b}\in B_{n-1}^\mathrm{op}(k_F)} B_{\tau,\psi_1}(\overline{b})
   \end{align*}
   where the second equality follows from \Cref{thm PS main thm}. Finally, we observe that the sum $\sum_{\overline{b}\in B_{n-1}^\mathrm{op}(k_F)} B_{\tau,\psi_1}(\overline{b})$ is non-zero by \cite[\S $2.2$(B$1$)]{girsch2025newformscuspidalrepresentations}. This concludes the proof of \Cref{prop: non-vanishing for general n}.
    \end{proof}
\end{prop}
\begin{proof}[Proof of \Cref{thm PS in terms of new}]
    By combining \eqref{eq: integral simplified} with \Cref{prop: non-vanishing for general n}, we deduce that $\mathbf{W}\neq 0$. Then, the result follows by \Cref{lem: reduction to W non zero}.
\end{proof}
\section{Integral representations in positive depth}\label{sec: pos. depth}
Let $V$ be an $F$-vector space of dimension $n$. Let $G:=\GL_F(V)\simeq G_n$ and $A:= \mathrm{End}_F(V)\simeq M_{n,n}(F).$ 
\subsection{Hereditary orders and minimax strata}\label{sec: hered. order}
Let $\Lambda$ be an $\mathcal{O}_F$-lattice chain associated to $V$, and let $\mathfrak{A}_0(\Lambda)\subseteq A$ be the associated \textit{hereditary $\mathcal{O}_F$-order}. The order $\mathfrak{A}_0(\Lambda)$ comes equipped with a decreasing filtration $\{\mathfrak{A}_r(\Lambda)\}_{r\in\mathbf{Z}}$ by $\mathcal{O}_F$-submodules, and $\mathfrak{P}=\mathfrak{P}(\Lambda):=\mathfrak{A}_1(\Lambda)$ is the Jacobson radical of $\mathfrak{A}_0(\Lambda)$. We define subgroups of $G$ by
$$U^0(\Lambda):=\mathfrak{A}_0(\Lambda)^\times,\ \ U^r(\Lambda):=1+\mathfrak{A}_r(\Lambda),$$
for $r\geq 1.$ For a concise but more detailed description of these objects we refer the reader to \cite[\S $6.1$]{girsch2025newformscuspidalrepresentations}.

A \textit{stratum} in $A$ is a $4$-tuple $[\Lambda,m,m',\beta]$ where 
\begin{enumerate}
    \item $\Lambda$ is an $\mathcal{O}_F$-lattice chain associated to $V$,
    \item $m>m'$ are integers,
    \item $\beta\in \mathfrak{P}^{-m}\subseteq A$.
\end{enumerate}
A stratum $[\Lambda,m,m-1,\beta]$ is called \textit{fundamental} if $\beta+\mathfrak{P}^{-m+1}$ contains no nilpotent elements.
\begin{lem}\emph{(\cite{BushnellKutzko},\cite{girsch2025newformscuspidalrepresentations})}\label{lem: crucial iso}
    For integers $0\leq a' \leq a\leq 2a'+1$, we have an isomorphism
    \begin{align*}\mathfrak{P}^{-a}/\mathfrak{P}^{-a'}&\simeq \mathrm{Hom}(U^{a'+1}(\Lambda)/U^{a+1}(\Lambda),\mathbf{C}^\times)\\
     b+ \mathfrak{P}^{-a'}&\mapsto\left[\psi_b:x \mapsto \psi_1(\mathrm{Tr}_{A/F}(b(x-1)))\right].\end{align*}
\end{lem}
If $\pi$ is an irreducible representation of $G$, which is not of depth zero, then then there exists a fundamental stratum $[\Lambda,m,m-1,\beta]$ with $m\geq 1$ such that 
\begin{align}\label{eq: psi_beta}
    \mathrm{Hom}_{U^m(\Lambda)}(\psi_\beta,\pi)\neq 0,
\end{align}
where $\psi_\beta$ is regarded as a character of $U^m(\Lambda)$ via \Cref{lem: crucial iso} with $(a',a)=(m-1,m)$.
Following the work of Bushnell-Kutzko \cite{BushnellKutzko}, if \eqref{eq: psi_beta} holds we say that $\pi$ contains the fundamental stratum $[\Lambda,m,m-1,\beta]$. The authors of \textit{op.cit.} also show that if $\pi$ is cuspidal, then it contains a particularly nice fundamental stratum $[\Lambda,m,m-1,\beta]$, in the sense that it satisfies the following conditions:
\begin{enumerate}
    \item The $F$-algebra $E:=F[\beta]\subseteq A$ is a field,
    \item $\Lambda$ (which is, a priori, an $\mathcal{O}_F$-lattice chain) is an $\mathcal{O}_E$-lattice chain,
    \item The element $\beta$ is minimal over $F$, meaning:
    \begin{itemize}
        \item $\mathrm{gcd}(m,e(E/F))=1$,
        \item The element $\varpi_F^m \beta^{e(E/F)}$ generates the residue field extension $k_E/k_F$.
    \end{itemize}
\end{enumerate}
A fundamental stratum satisfying the above conditions is called a \textit{minimal} stratum.

A \textit{minimax} stratum is a minimal stratum $[\Lambda,m,m-1,\beta]$ which is also maximal in that $\mathrm{dim}_F(E)=\mathrm{dim}_F(F[\beta])=n$.
\begin{defn}\emph{(\cite[Definition $6.2$]{girsch2025newformscuspidalrepresentations})}
    A positive depth cuspidal representation of $G$ is called minimax if it contains a minimax stratum $[\Lambda,m,m-1,\beta]$ such that $$\mathrm{Hom}_{U^{\lf m/2 \rf+1}(\Lambda)}(\psi_\beta,\pi)\neq 0,$$
    where $\psi_\beta$ is regarded as a character of $U^{\lf m/2 \rf+1}(\Lambda)$ via \emph{\Cref{lem: crucial iso}} with $(a',a)=(\lf m/2 \rf, m)$.
    \emph{(}In other words, it also contains the simple stratum $[\Lambda,m,0,\beta].$\emph{)}
\end{defn}
\subsection{Setup}\label{sec: setup}
From now on, our setup is as follows. Let $[\Lambda,m,0,\beta]$ be a simple stratum in $A$ such that  $[\Lambda,m,m-1,\beta]$ is a minimax stratum in $A$, and the field extension $E/F$ is unramified, i.e. $e(E/F)=1$. As in \cite[\S $6.4$]{girsch2025newformscuspidalrepresentations}, we identify, once and for all, the $F$-vector space $V$ with $E=F[\beta]$ via the choice of the following ordered basis
$$\mathcal{B}':=\{\varpi_F^{km}\beta^k\ |\ 0\leq k\leq n-1\}.$$ 
\begin{rem}\label{rem: basis B'}
With respect to this basis, $\mathfrak{A}_0(\Lambda)=M_{n,n}(\mathcal{O}_F), \mathfrak{P}=M_{n,n}(\mathfrak{p}_F)$, and $U^r(\Lambda)=1+M_{n,n}(\mathfrak{p}_F^r)$ for all $r\geq 1.$ Moreover, in this basis, the character $\psi_\beta$ agrees on upper triangular unipotent matrices contained in $U_n(F)$ with the non-degenerate character $\Psi_m$ from \Cref{sec: add. chars}. This follows from the explicit description of $\psi_\beta$ in this basis given in \cite[\S $6.4$]{girsch2025newformscuspidalrepresentations}. Note that there is a typo in \textit{loc.cit.} as their ``$\psi^{\mathbf{t}_m}$'' should read ``$\psi^{\mathbf{t}_m^{-1}}$''. 
\end{rem}
\subsection{Simple characters and compact induction}\label{sec: simple chars}
Using the notation of \cite{girsch2025newformscuspidalrepresentations}, we define the following open subgroups of $G$:
\begin{itemize}
    \item $H^1:=(1+\mathfrak{p}_E)U^{\lf m/2 \rf+1}(\Lambda),$
    \item $J^1:=(1+\mathfrak{p}_E)U^{\lf (m+1)/2 \rf}(\Lambda)$,
    \item $J:=\mathcal{O}_E^\times U^{\lf (m+1)/2 \rf}(\Lambda),$
    \item $\mathbf{J}:=E^\times U^{\lf (m+1)/2 \rf}(\Lambda).$
\end{itemize}
We have the chain of normal inclusions $H^1\subseteq J^1\subseteq J \subseteq \mathbf{J}$. The set of \textit{simple characters} of $H^1$ associated to the simple stratum $[\Lambda,m,0,\beta]$ is
$$\mathcal{C}(\Lambda,\beta):=\left\{\theta\in \mathrm{Hom}(H^1,\mathbf{C}^\times)\ |\ \theta|_{U^{\lf m/2 \rf+1}(\Lambda)}=\psi_\beta\right\}.$$
Given a simple character $\theta\in \mathcal{C}(\Lambda,\beta)$, we can associate to it the character 
\begin{align}\label{eq: character}
    \theta_{\Psi_m}:(U_n(F)\cap J^1)H^1\rightarrow \mathbf{C}^\times,\ uh\mapsto \Psi_m(u)\theta(h)
\end{align}
which is well-defined since $U_n(F)\cap H^1=U_n(F)\cap U^{\lf m/2 \rf+1}(\Lambda)$ and both $\Psi_m$ and $\theta$ agree on this intersection by the second half of \Cref{rem: basis B'}. Then, we set 
$$\eta_{\theta}:=\mathrm{c}\text{-}\mathrm{Ind}_{(U_n(F)\cap J^1)H^1}^{J^1}\left(\theta_{\Psi_m}\right).$$
The representation $\eta_\theta$, which is a priori a representation of $J^1,$ extends to a representation of $\mathbf{J}$. Choose any extension $\lambda$ of $\eta_\theta$. Then the representation of $G$ given by 
$$\pi:=\mathrm{c}\text{-}\mathrm{Ind}_{\mathbf{J}}^G(\lambda)$$
is irreducible, cuspidal, minimax, and of integral depth. Moreover, all irreducible cuspidal minimax representations of integral depth arise this way. Finally, by Bushnell's conductor formula \eqref{eq: Bushnel conductor formula}, $c(\pi)=n(m+1).$
\subsection{Whittaker newforms as averages of  Pa\v{s}k\={u}nas-Stevens Whittaker functions}\label{sec: PS in terms of new pos depth}
Let $\pi=\mathrm{c}\text{-}\mathrm{Ind}_{\mathbf{J}}^G(\lambda)$ be an irreducible cuspidal minimax representation of $G_n$ with integral depth, constructed as in \Cref{sec: simple chars}. As in the depth-zero construction of \Cref{sec: PS in depth zero} but with some modifications, there exists a Pa\v{s}k\={u}nas-Stevens Whittaker function $W_{\pi,\Psi_m}^\mathrm{PS}$  characterized as the unique, up to scalars, non-zero element of the one-dimensional $((U_n(F)\cap J^1)H^1,\theta_{\Psi_m})$-isotypic subspace of the Whittaker model $\mathcal{W}(\pi,\Psi_m).$ The explicit construction of $W_{\pi,\Psi_m}^\mathrm{PS}$ is via a unipotent integral transform similar to the one given in \Cref{sec: PS in depth zero}. We refer the reader to \cite[Theorem $5.8$]{PaskunasStevens2008} for more details.
\begin{thm}\emph{(\cite[Theorem $6.5$ \& Corollary $6.8$] {girsch2025newformscuspidalrepresentations})}\label{thm: new interms of PS minimax} Let $\Lambda,m,\beta,\mathcal{B}'$ be as in \emph{\Cref{sec: setup}}. Let $\pi=\mathrm{c}\text{-}\mathrm{Ind}_{\mathbf{J}}^{G_n}(\lambda)$ be an irreducible cuspidal minimax representation of $G_n$ with integral depth, constructed as in \emph{\Cref{sec: simple chars}}. The Whittaker newform $W_{\pi,\psi_0}^\mathrm{new}\in \mathcal{W}(\pi,\psi_0)$ admits the following integral representation in terms of the Pa\v{s}k\={u}nas-Stevens Whittaker function $W_{\pi,\Psi_m}^\mathrm{PS}\in \mathcal{W}(\pi,\Psi_m)$:
    $$W_{\pi,\psi_0}^\mathrm{new}(g)=\int_{K_n(n(m+1))} W_{\pi,\Psi_m}^\mathrm{PS}(\Sigma_{m,n} gk \Sigma_{m,n}^{-1})\ dk,\ \mathrm{for\ all\ }g\in G_n,$$
    for an appropriately normalized Haar measure $dk$ on $K_n(n(m+1)).$
\end{thm}
\subsection{Pa\v{s}k\={u}nas-Stevens Whittaker functions as averages of Whittaker newforms}
Analogously to \Cref{PS interms of new}, the goal of this section is to establish the reverse averaging formula to the one in \Cref{thm: new interms of PS minimax}.
\begin{thm}\label{thm: PS in terms of new minimax}
    Let $\Lambda,m,\beta,\mathcal{B}'$ be as in \emph{\Cref{sec: setup}}, and the character $\theta_{\Psi_m}$ be as in \eqref{eq: character}. Let $\pi=\mathrm{c}\text{-}\mathrm{Ind}_{\mathbf{J}}^{G_n}(\lambda)$ be an irreducible cuspidal minimax representation of $G_n$ with integral depth, constructed as in \emph{\Cref{sec: simple chars}}. The Pa\v{s}k\={u}nas-Stevens Whittaker function $W_{\pi,\Psi_m}^\mathrm{PS}\in \mathcal{W}(\pi,\Psi_m)$ admits the following integral representation in terms of the Whittaker newform $W_{\pi,\psi_0}^\mathrm{new}\in \mathcal{W}(\pi,\psi_0)$:
    $$W_{\pi,\Psi_m}^\mathrm{PS}(g)=\int_{(U_n(F)\cap J^1)H^1}W_{\pi,\psi_0}^\mathrm{new}(\Sigma_{m,n}^{-1} gx \Sigma_{m,n})\theta_{\Psi_m}(x^{-1})\ dx,\ \mathrm{for\ all}\ g\in G_n,$$
    for an appropriately normalized Haar measure $dx$ on $(U_n(F)\cap J^1)H^1$.
\end{thm}
The remainder of this section is dedicated to proving the above theorem. We introduce the function $\mathbf{W}_m:G_n\rightarrow \mathbf{C}$ given by
$$\mathbf{W}_m(g):=\int_{(U_n(F)\cap J^1)H^1}W_{\pi,\psi_0}^\mathrm{new}(\Sigma_{m,n}^{-1} gx \Sigma_{m,n})\theta_{\Psi_m}(x^{-1})\ dx,\ g\in G_n.$$
\begin{lem}\label{lem: reduction to W non-zero minimax}
    The function $\mathbf{W}_m$ is an element of the Whittaker model $\mathcal{W}(\pi,\Psi_m)$. Moreover, if $\mathbf{W}_m\neq 0$, then \emph{\Cref{thm: new interms of PS minimax}} holds.
    \begin{proof}
        The proof is exactly the same as that of \Cref{lem: W is in Whit model} and \Cref{lem: reduction to W non zero}. For the second part we use the characterizing property of $W_{\pi,\Psi_m}^\mathrm{PS}$ as the unique, up to scalars, non-zero element of the one-dimensional space given by the $((U_n(F)\cap J^1)H^1,\theta_{\Psi_m})$-isotypic component of $\mathcal{W}(\pi,\Psi_m).$
    \end{proof}
\end{lem}
\subsubsection{Non-vanishing of $\mathbf{W}_m(1)$} By the above lemma, the proof of \Cref{thm: new interms of PS minimax} reduces to showing that $\mathbf{W}_m(1)\neq 0$. Since $U_n(F)\cap J^1$ normalizes $H^1$, we can unfold the integral $\mathbf{W}_m(1)$ in exactly the same way as in the depth zero case, by using \cite[Lemma $2.4$]{kurinczuk2017rankin}. Thus, 
$$\mathbf{W}_m(1)=\int_{U_n(F)\cap J^1}\int_{H^1} W_{\pi,\psi_0}^\mathrm{new}(\Sigma_{m,n}^{-1}u h \Sigma_{m,n})\theta_{\Psi_m}(uh)^{-1}\ du\ dh,$$
for appropriately normalized Haar measures $du$ on $U_n(F)\cap J^1$ and $dh$ on $H^1.$ Now notice that $\psi_0(\Sigma_{m,n}^{-1}u\Sigma_{m,n})=\theta_{\Psi_m}(u).$ Thus, 
\begin{align}\label{eq: reduction pos. depth}
    \mathbf{W}_m(1)=\mathrm{vol}(U_n(F)\cap J^1,du)\cdot \int_{H^1}W_{\pi,\psi_0}^\mathrm{new}(\Sigma_{m,n}^{-1}h \Sigma_{m,n})\theta_{\Psi_m}(h^{-1})\ dh.
\end{align}
In order to show that $\mathbf{W}_m(1)\neq 0,$ it suffices to show that the integral over $H^1$ appearing in \eqref{eq: reduction pos. depth} is non-zero.
\begin{prop}\label{prop: main in minimax}
    With the setup as in the statement of \emph{\Cref{thm: PS in terms of new minimax}}, we have 
    \begin{align}\label{eq: integral in q.}\int_{H^1}W_{\pi,\psi_0}^\mathrm{new}(\Sigma_{m,n}^{-1}h \Sigma_{m,n})\theta_{\Psi_m}(h^{-1})\ dh\neq 0.\end{align}
    \begin{proof}
        Start by noting that $\theta_{\Psi_m}(h^{-1})=\theta(h^{-1})$ for all $h\in H^1$ by definition. By \Cref{thm: new interms of PS minimax}, we have \begin{align}\label{eq: beginning of proof.}\nonumber\int_{H^1}W_{\pi,\psi_0}^\mathrm{new}(\Sigma_{m,n}^{-1}h \Sigma_{m,n})\theta_{\Psi_m}(h^{-1})\ dh&=\int_{H^1}\int_{K_n(n(m+1))}W_{\pi,\Psi_m}^\mathrm{PS}(h \Sigma_{m,n}k \Sigma_{m,n}^{-1})\theta(h^{-1})\ dk\ dh\\
        &=\int_{K_n(n(m+1))} \mathbf{I}(\Sigma_{m,n} k \Sigma_{m,n}^{-1})\ dk,\end{align}
        where $$\mathbf{I}:G_n\rightarrow \mathbf{C},\ \mathbf{I}(g):=\int_{H^1} W_{\pi,\Psi_m}^\mathrm{PS}(hg)\theta (h^{-1})\ dh,\ g\in G_n.$$
        We now prove the following technical lemma regarding the auxiliary integral $\mathbf{I}$. 
        \begin{lem}\label{lem: lemma about I}
            The following hold:
            \begin{enumerate}
            \item The integral $\mathbf{I}$ satisfies the bi-equivariance property
            $$\mathbf{I}(g_1gg_2)=\theta_{\Psi_m}(g_1)\theta_{\Psi_m}(g_2)\mathbf{I}(g),\ \mathrm{for\ all}\ g_1,g_2\in (U_n(F)\cap J^1)H^1,\ g\in G_n. $$
                \item Let $g\in G_n.$
                $$\mathbf{I}(g)\neq 0\ \implies\ g\in I_{G_n}(\theta_{\Psi_m}):=\left\{g \in G_n\ |\ \substack{\theta_{\Psi_m}(\gamma)=\theta_{\Psi_m}(g\gamma g^{-1})\ \\ \forall\ \gamma \in (U_n(F)\cap J^1)H^1\cap g^{-1}(U_n(F)\cap J^1)H^1 g}\right\}.$$
                \item The support $\mathrm{supp}(\mathbf{I})\subseteq \mathbf{J}$, where the subgroup $\mathbf{J}\subseteq G_n$ is given at the beginning of \emph{\Cref{sec: simple chars}}.
            \end{enumerate}
            \begin{proof}
                We start with part $(1)$. Let $g_1,g_2\in (U_n(F)\cap J^1)H^1.$ Since $W_{\pi,\Psi_m}^\mathrm{PS}$ lies in the $((U_n(F)\cap J^1)H^1,\theta_{\Psi_m})$-isotypic subspace of $\mathcal{W}(\pi,\Psi_m)$, we immediately see that $\mathbf{I}(g_1gg_2)=\theta_{\Psi_m}(g_2)\mathbf{I}(g_1g).$ Write $g_1=n_1h_1$ with $n_1\in U_n(F)\cap J^1$ and $h_1\in H^1$. Then, 
                \begin{align*}\mathbf{I}(g_1g)&=\int_{H^1} W_{\pi,\Psi_m}^\mathrm{PS}(hn_1h_1g)\theta(h^{-1})\ dh\\
                &=\theta(n_1h_1 n_1^{-1})\int_{H^1}W_{\pi,\Psi_m}^\mathrm{PS}(hn_1g)\theta(h^{-1})\ dh\\
                &=\theta(h_1)\int_{H^1}W_{\pi,\Psi_m}^\mathrm{PS}(hn_1g)\theta(h^{-1})\ dh\\
                &=\theta(h_1)\Psi_m(n_1)\int_{H^1}W_{\pi,\Psi_m}^\mathrm{PS}(hg)\theta((n_1hn_1^{-1})^{-1})\ dh\\
                &=\theta_{\Psi_m}(g_1)\mathbf{I}(g).\end{align*}
                 The second equality follows from the fact that $n_1h_1n_1^{-1}\in H^1$ together with a change of variables $h\rightsquigarrow h(n_1h_1n_1^{-1})$; the third equality follows, again, from the fact that $n_1h_1n_1^{-1}\in H^1$, and the fact that $\theta=\theta_{\Psi_m}|_{H^1}$; the fourth equality follows from the left equivariance of the $\Psi_m$-Whittaker function and a change of variables $h\rightsquigarrow n_1^{-1}hn_1\in H^1$, and the last equality follows, again, from the fact that $\theta=\theta_{\Psi_m}|_{H^1}$. This concludes the proof of part $(1)$.

                 For part $(2),$ assume that $\mathbf{I}(g)\neq 0$ for $g\in G_n$. Let $\gamma$ be any element of $(U_n(F)\cap J^1)H^1\cap g^{-1}(U_n(F)\cap J^1)H^1g$. Then, we have the chain of equalities 
                 \begin{align*}
                     \theta_{\Psi_m}(\gamma)\mathbf{I}(g)&=\mathbf{I}(g\gamma )\\
                     &=\mathbf{I}(g\gamma g^{-1} g)\\
                     &=\theta_{\Psi_m}(g\gamma g^{-1})\mathbf{I}(g),
                 \end{align*}
                 where the first equality follows from part $(1)$, and the last equality follows from the fact that $g\gamma g^{-1} \in (U_n(F)\cap J^1)H^1$ together with part $(1).$ Since by assumption $\mathbf{I}(g)\neq 0$, we get that $\theta_{\Psi_m}(\gamma)=\theta_{\Psi_m}(g\gamma g^{-1})$ for all $\gamma \in (U_n(F)\cap J^1)H^1\cap g^{-1}(U_n(F)\cap J^1)H^1g$, which in turn implies that $g\in I_{G_n}(\theta_{\Psi_m})$ as claimed.

                 Finally, we prove part $(3)$. Since $\theta_{\Psi_m}|_{H^1}=\theta,$ we see that $$I_{G_n}(\theta_{\Psi_m})\subseteq I_{G_n}(\theta):=\left\{g \in G_n\ |\ \substack{\theta(h)=\theta(gh g^{-1})\ \\\forall\ h \in H^1\cap g^{-1}H^1 g}\right\}.$$
                 Combining this with part $(2)$, we see that $\mathrm{supp(\mathbf{I})}\subseteq I_{G_n}(\theta)$. But by \cite[($3.3.2$)]{BushnellKutzko}, applied to this integral-depth minimax setting, we have that $I_{G_n}(\theta)=\mathbf{J}$. This proves part $(3)$ and concludes the proof of the lemma.
            \end{proof}
        \end{lem}
        We now return to the proof of \Cref{prop: main in minimax}. We know from \Cref{lem: lemma about I}$(3)$ that $\mathbf{I}$ vanishes away from $\mathbf{J}$. Now, let $j\in \mathbf{J}$, and let $\mathcal{J}=\mathcal{J}_\lambda:\mathbf{J}\rightarrow \mathbf{C}$ denote the Pa\v{s}k\={u}nas-Stevens Bessel function associated to $\lambda$ as in \cite[\S $5.1$]{PaskunasStevens2008}. Then, 
        \begin{align*}\mathbf{I}(j)&=\int_{H^1} W_{\pi,\Psi_m}^\mathrm{PS}(hj)\theta(h^{-1})\ dh\\
        &=\int_{H^1}\mathcal{J}(hj)\theta(h^{-1})\ dh\\
        &=\int_{H^1}\mathcal{J}(j)\theta_{\Psi_m}(h)\theta(h^{-1})\ dh\\
        &=\mathrm{vol}(H^1,dh) \mathcal{J}(j),
        \end{align*}
        where the second equality follows from the ``moreover'' part of \cite[Theorem $5.8$]{PaskunasStevens2008}, and the third equality follows from \cite[Proposition $5.3$(iii)]{PaskunasStevens2008}. Thus, the auxiliary integral $\mathbf{I}$ vanishes away from $\mathbf{J}$, and is equal to the Pa\v{s}k\={u}nas-Stevens Bessel function $\mathcal{J}$ on $\mathbf{J}$, up to a non-zero scalar. As in \cite[Theorem $6.5(2)$]{girsch2025newformscuspidalrepresentations}, we identify $\lambda$ with the model $\mathcal{W}(\lambda,\theta_{\Psi_m})\subseteq \mathrm{c}\text{-}\mathrm{Ind}_{U_n(F)\cap J^1)H^1}^\mathbf{J}(\theta_{\Psi_m})$, which contains the Bessel function $\mathcal{J}$. Then, we consider the associated Bessel vector 
        \begin{align*}
            \mathscr{J}\in \pi\simeq \mathrm{c}\text{-}\mathrm{Ind}_\mathbf{J}^{G_n} \lambda,\ \mathscr{J}(g):=\begin{dcases}
            \lambda(g)\mathcal{J},&\ \text{if}\ g\in \mathbf{J},\\
            0,&\ \text{otherwise}
            \end{dcases}
        \end{align*}
        which by construction satisfies $$\mathscr{J}(g)(1)=\begin{dcases}\mathcal{J}(g),\ &\mathrm{if}\ g\in \mathbf{J},\\
        0,\ &\mathrm{otherwise}\end{dcases}.$$
        Hence, up to a non-zero scalar (which we ommit from the notation), $$\mathbf{I}(g)=\mathscr{J}(g)(1),\ \mathrm{for\ all\ } g\in G_n.$$ Using this, \eqref{eq: beginning of proof.} implies that the integral in question (i.e. \eqref{eq: integral in q.}) can be written as 
        \begin{align}\label{eq: rewritten}
            \int_{H^1}W_{\pi,\psi_0}^\mathrm{new}(\Sigma_{m,n}^{-1}h \Sigma_{m,n})\theta_{\Psi_m}(h^{-1})\ dh=\int_{K_n(n(m+1))} \mathscr{J}(\Sigma_{m,n}k \Sigma_{m,n}^{-1})\ dk(1).
        \end{align}
      Finally, we need the following result of \cite{girsch2025newformscuspidalrepresentations}.
      \begin{thm}\emph{(\cite[Theorem $6.5$(2)]{girsch2025newformscuspidalrepresentations})}\label{eq: thm gk 6.5(2)}
          Let the setup be as in the statement of \emph{\Cref{thm: PS in terms of new minimax}}, and let $f^\mathrm{new}\in \mathrm{c}\text{-}\mathrm{Ind}_\mathbf{J}^{G_n}(\lambda)\simeq \pi$ denote the unique, up to scalars, newform. Then, 
          \begin{enumerate}
              \item $f^\mathrm{new}(\Sigma_{m,n})(1)=1$.
              \item $f^\mathrm{new}(g)=\displaystyle\int_{K_n(n(m+1))}\mathscr{J}(gk\Sigma_{m,n}^{-1})\ dk$ for an appropriately normalized Haar measure $dk$ on $K_n(n(m+1))$.
          \end{enumerate}
      \end{thm}
      Combining \eqref{eq: rewritten} with \Cref{eq: thm gk 6.5(2)} gives us the non-vanishing of \eqref{eq: integral in q.}. This concludes the proof of \Cref{prop: main in minimax}.
    \end{proof}
\end{prop}
\begin{proof}[Proof of \Cref{thm: PS in terms of new minimax}] \Cref{prop: main in minimax} implies that $\mathbf{W}_m\neq 0$. The result then follows from \Cref{lem: reduction to W non-zero minimax}.
    
\end{proof}
\section{Explicit test vectors for $\GL(n)\times \GL(n)$}
Let $\pi_1,\pi_2$ be irreducible generic  representations of $G_n$ and write $\mathcal{S}(F^n)$ for the space of complex-valued Schwartz functions on $F^n$. Jacquet-Piatetski-Shapiro-Shalika \cite{jacquet1983rankin} construct a local zeta integral associated to a pair of such representations, whose definition we now recall. 

Let $\psi:F\rightarrow\mathbf{C}^\times$ be any non-trivial additive character. Let $W_1\in \mathcal{W}(\pi_1,\psi), W_2\in \mathcal{W}(\pi_2,\psi^{-1})$, $\Phi\in \mathcal{S}(F^n)$, and $s$ a complex parameter. The zeta integral of \textit{op.cit.} is defined as
\begin{align}\label{eq: RS integral}
    Z(W_1,W_2,\Phi,s):=\int_{U_n(F)\backslash \GL_n(F)} W_1(g)W_2(g)\Phi(e_ng)|\det(g)|^s\ dg
\end{align}
where $e_n:=(0,\dots,0,1)\in F^n$ and $dg$ is the unique, up to scalars, quotient measure on $U_n(F)\backslash \GL_n(F)$. By the seminal work of $\textit{op.cit.}$ the above integral converges for $\Re(s)$ large enough, and admits a unique meromorphic continuation as a rational function of $q^s$. Moreover, there exists an $L$-factor $L(\pi_1,\pi_2,s)$ only depending on $\pi_1,\pi_2$, for which the fractional ideal $$\{ Z(W_1,W_2,\Phi,s)\ |\ W_1\in \mathcal{W}(\pi_1,\psi), W_2\in \mathcal{W}(\pi_2,\psi^{-1}), \Phi\in \mathcal{S}(F^n) \}$$ of $\mathbf{C}[q^{\pm s}]$ is precisely $L(\pi_1,\pi_2,s) \mathbf{C}[q^{\pm s}]\subseteq \mathbf{C}(q^s)$. Thus, there exists a finite number of Whittaker functions $W_{1,i}\in \mathcal{W}(\pi_1,\psi)$, $W_{2,i}\in \mathcal{W}(\pi_2,\psi^{-1})$, and a finite number of Schwartz functions $\Phi_i\in \mathcal{S}(F^n)$ such that 
$$L(\pi_1,\pi_2,s)=\sum_i Z(W_{1,i},W_{2,i},\Phi_i,s).$$
 
  Anatural question that arises is whether one can find explicit families of such \textit{test vectors}, representing the $L$-factor $L(\pi_1,\pi_2,s).$ For example, for $\GL(2)\times \GL(2)$, this question is addressed in detail in the PhD thesis of Kim \cite{kim2010test} where explicit test vectors for most pairs of representations are obtained. For $\GL(n)\times \GL(n)$, \textit{op.cit.} also addresses this question under the assumption that at least of one the $\pi_i$ is unramified. For cuspidal pairs on $\GL(n)\times\GL(n)$, Kurinczuk-Matringe \cite{TestvectorsRobNadir} give explicit test vectors in terms of Pa\v{s}k\={u}nas-Stevens Whittaker functions, whenever $L(\pi_1,\pi_2,s)$ is not identically equal to $1$. In other words, whenever $\pi_1$ is an unramified twist of $\pi_2^\vee$. Their main result is the following.
\begin{thm}\emph{(\cite[Theorem $9.1$]{TestvectorsRobNadir})}\label{thm: main theorem of KurinczukMatringe} Let $\pi$ be an irreducible cuspidal representation of $G_n$, and $\chi$ an unramified character of $F^\times$. Consider the following input data:
\begin{itemize}
    \item The Pa\v{s}k\={u}nas-Stevens Whittaker functions $W_{\chi\pi,\psi}^\mathrm{PS}\in \mathcal{W}(\chi\pi,\psi)$ and $W_{\pi^\vee,\psi^{-1}}^\mathrm{PS}\in \mathcal{W}(\pi^\vee,\psi^{-1})$, as constructed in \cite{PaskunasStevens2008}, with respect
to a suitable non-degenerate character $\psi$ of $U_n(F)$
and suitable maximal extended simple types. 
\item The Schwartz function $\Phi_0:=\mathbf{1}_{\mathcal{O}_F^n}\in\mathcal{S}(F^n).$
\end{itemize}
Then, up to appropriate normalization of the quotient measure $dg$ appearing in \eqref{eq: RS integral}, we have 
$$L(\chi\pi,\pi^\vee,s)=Z(W_{\chi\pi,\psi}^\mathrm{PS},W_{\pi^\vee,\psi^{-1}}^\mathrm{PS},\Phi_0,s).$$

\end{thm}

In this section, we combine the results of \Cref{sec: depth zero} and \Cref{sec: pos. depth} with the work of \cite{TestvectorsRobNadir}, to give new explicit test vectors in various cuspidal families on $\GL(n)\times\GL(n)$ which are given in terms of \textit{Whittaker newforms}. In particular, we do this for $L$-factors $L(\chi\pi,\pi^\vee,s)$ for a cuspidal $\pi$ on $\GL(n)$ and an unramified character $\chi$ on $F^\times$, as long as $\pi$ is either depth-zero or minimax of integral depth. 
\subsection{Depth zero}

\begin{thm}\label{thm test vector depth zero}
    Let $\pi$ be a depth-zero irreducible cuspidal representation of $G_n$, and $\chi$ an unramified character of $F^\times$. Consider the following input data:
    \begin{itemize}
        \item The Whittaker newform $W_{\chi \pi,\psi_0}^\mathrm{new}\in \mathcal{W}(\chi\pi,\psi_0)$,
        \item The average $\widetilde{W}_{\pi^\vee,\psi_0^{-1}}^\mathrm{new}\in \mathcal{W}(\pi^\vee,\psi_0^{-1})$ of the dual Whittaker newform, 
$$\widetilde{W}_{\pi^\vee,\psi_0^{-1}}^\mathrm{new}:=\int_{U_n(\mathcal{O}_F)K_n^1} W_{\pi^\vee,\psi_0^{-1}}^\mathrm{new}( (-)\Sigma_n^{-1}y\Sigma_n) \psi_1(y)\ dy,$$
        \item The Schwartz function $\Phi_n:=\mathbf{1}_{\mathcal{O}_F^n\cdot \Sigma_n}\in \mathcal{S}(F^n)$.
    \end{itemize}
    Then, up to appropriate normalization of the quotient measure $dg$ appearing in \eqref{eq: RS integral}, we have 
    $$L(\chi\pi,\pi^\vee,s)=Z(W_{\chi \pi,\psi_0}^\mathrm{new},\widetilde{W}_{\pi^\vee,\psi_0^{-1}}^\mathrm{new}, \Phi_n,s). $$
\begin{proof}
    For ease of notation, we will assume that $\chi$ is the trivial character. The proof for general unramified $\chi$ is identical upon carrying over all the unramified twists of Whittaker functions. 

    Let $W_{\pi,\psi_1}^\mathrm{PS}\in \mathcal{W}(\pi,\psi_1)$ and $W_{\pi^\vee,\psi_1^{-1}}^\mathrm{PS}\in \mathcal{W}(\pi^\vee, \psi_1^{-1})$ be the associated depth-zero Paskunas-Stevens Whittaker functions already studied in \Cref{sec: depth zero}. Let $\Phi_0:=\mathbf{1}_{\mathcal{O}_F^n}$. We start the proof by applying the main result of \cite{TestvectorsRobNadir} (i.e. \Cref{thm: main theorem of KurinczukMatringe}) to this depth-zero setting. In particular, \Cref{thm: main theorem of KurinczukMatringe} tells us that up to appropriate normalization of the quotient measure $dg$ appearing in \eqref{eq: RS integral}, we have 
        $$L(\pi, \pi^\vee,s)=Z(W_{\pi,\psi_1}^\mathrm{PS},W_{\pi^\vee,\psi_1^{-1}}^\mathrm{PS}, \Phi_0).$$
Now, we combine this with the definition of the zeta integral \eqref{eq: RS integral} and \Cref{thm PS in terms of new} to obtain the following: 
\begin{align*}
L(\pi,\pi^\vee,s)=&\int_{U_n(F)\backslash \GL_n(F)} \left(\int_{U_n(\mathcal{O}_F)K_n^1} W_{\pi,\psi_0}^\mathrm{new}(\Sigma_n^{-1}gx \Sigma_n)\psi_1(x^{-1})\ dx\right)\\
    &\times\left(\int_{U_n(\mathcal{O}_F)K_n^1} W_{\pi^\vee,\psi_0^{-1}}^\mathrm{new}(\Sigma_n^{-1} gy \Sigma_n) \psi_1^{-1}(y^{-1})\ dy\right)\Phi_0(e_ng)|\det(g)|^s\ dg.
\end{align*} 
Then, the change of variables $g\rightsquigarrow gx$, together with the fact that $U_n(\mathcal{O}_F)K_n^1\subseteq K_n$ stabilizes $\Phi_0$, imply that
\begin{align*}
L(\pi,\pi^\vee,s)=&\int_{U_n(F)\backslash \GL_n(F)} W_{\pi,\psi_0}^\mathrm{new}(\Sigma_n^{-1}g\Sigma_n)\Phi_0(e_ng)|\det(g)|^s\\
&\times\int_{U_n(\mathcal{O}_F)K_n^1}\int_{U_n(\mathcal{O}_F)K_n^1} W_{\pi^\vee,\psi_0^{-1}}^\mathrm{new}(\Sigma_n^{-1} gx^{-1}y \Sigma_n) \psi_1(x^{-1}y)\ dx\ dy\ dg.
\end{align*}
Now, the change of variables $y\rightsquigarrow x^{-1}y$ followed by the change of variables $g\rightsquigarrow \Sigma_n^{-1} g \Sigma_n$ (\textit{cf.} \cite[\S $2.12(1)$]{jacquet1983rankin}) and the fact that $e_n\Sigma_n=e_n$ imply that 
\begin{align*}
    L(\pi,\pi^\vee,s)&=\mathrm{vol}(U_n(\mathcal{O}_F)K_n^1,dx)\delta_{B_n}(\Sigma_n^{-1})\int_{U_n(F)\backslash \GL_n(F)} W_{\pi,\psi_0}^\mathrm{new}(g) \\
    &\times\left( \int_{U_n(\mathcal{O}_F)K_n^1} W_{\pi^\vee,\psi_0^{-1}}^\mathrm{new}( g\Sigma_n^{-1}y \Sigma_n) \psi_1(y)\ dy\right)\Phi_n(e_ng)|\det(g)|^s\ dg.\\
    &=\int_{U_n(F)\backslash \GL_n(F)} W_{\pi,\psi_0}^\mathrm{new}(g)\widetilde{W}_{\pi^\vee,\psi_0^{-1}}^\mathrm{new}(g) \Phi_n(e_ng)|\det(g)|^s\ dg \\
    &\times \mathrm{vol}(U_n(\mathcal{O}_F)K_n^1,dx)\delta_{B_n}(\Sigma_n^{-1}),
\end{align*}
where $\delta_{B_n}$ denotes the modulus character of $B_n.$
Thus, up to appropriate normalization of the quotient measure $dg$, we have 
$$L(\pi,\pi^\vee,s)=Z(W_{\pi,\psi_0}^\mathrm{new},\widetilde{W}_{\pi^\vee,\psi_0^{-1}}^\mathrm{new},\Phi_n,s)$$
as claimed.
\end{proof}
\end{thm}
\subsection{Positive depth}
\begin{thm}\label{thm: test vector pos. depth}
Let $\Lambda,m,\beta,\mathcal{B}'$ be as in \emph{\Cref{sec: setup}}, and the character $\theta_{\Psi_m}$ be as in \eqref{eq: character}. Let $\pi=\mathrm{c}\text{-}\mathrm{Ind}_{\mathbf{J}}^{G_n}(\lambda)$ be an irreducible cuspidal minimax representation of $G_n$ with integral depth, constructed as in \emph{\Cref{sec: simple chars}}. Let $\chi$ be an unramified character of $F^\times$ and consider the following input data:
 \begin{itemize}
        \item The Whittaker newform $W_{\chi \pi,\psi_0}^\mathrm{new}\in \mathcal{W}(\chi\pi,\psi_0)$,
        \item The average $\widetilde{W}_{\pi^\vee,\psi_0^{-1}}^\mathrm{new}\in \mathcal{W}(\pi^\vee,\psi_0^{-1})$ of the dual Whittaker newform, 
$$\widetilde{W}_{\pi^\vee,\psi_0^{-1}}^\mathrm{new}:=\int_{(U_n(F)\cap J^1)H^1} W_{\pi^\vee,\psi_0^{-1}}^\mathrm{new}( (-)\Sigma_{m,n}^{-1}y\Sigma_{m,n}) \theta_{\Psi_m}(y)\ dy,$$
        \item The Schwartz function $\Phi_{m,n}:=\mathbf{1}_{\mathcal{O}_F^n\cdot \Sigma_{m,n}}\in \mathcal{S}(F^n)$.
    \end{itemize}
    Then, up to appropriate normalization of the quotient measure $dg$ appearing in \eqref{eq: RS integral}, we have 
    $$L(\chi\pi,\pi^\vee,s)=Z(W_{\chi \pi,\psi_0}^\mathrm{new},\widetilde{W}_{\pi^\vee,\psi_0^{-1}}^\mathrm{new}, \Phi_{m,n},s).$$
    \begin{proof}
        The proof follows the same strategy as the proof of \Cref{thm test vector depth zero}, while utilizing the compatibility of duality with the inducing type-theoretic data used in the construction of $\pi$. The proof starts off by, once again, applying \Cref{thm: main theorem of KurinczukMatringe} to deduce that (assuming $\chi=1$ as before)
       $$L(\pi, \pi^\vee,s)=Z(W_{\pi,{\Psi_m}}^\mathrm{PS},W_{\pi^\vee,\Psi_m^{-1}}^\mathrm{PS}, \Phi_0,s).$$
       Then we apply the exact same integral manipulations as in the proof of \Cref{thm test vector depth zero} using \Cref{thm: PS in terms of new minimax}, and observing that in the basis $\mathcal{B}'$ the subgroup $(U_n(F)\cap J^1)H^1\subseteq \GL_n(\mathcal{O}_F)$ and thus stabilizes $\Phi_0$.
    \end{proof}
\end{thm}

\bibliography{citation} 
\bibliographystyle{amsalpha}
\address{Instituto Matemáticas de la Universidad de Sevilla, Seville 41012,
Spain}
\email{alexgroutides21@gmail.com}

\end{document}